\documentclass{amsart}
\usepackage{amsmath}
\usepackage{amsthm}
\usepackage{amsfonts}
\usepackage{amssymb}
\usepackage{amscd}
\usepackage[all]{xy}
\usepackage[percent]{overpic}
\usepackage{graphicx}
\usepackage{xcolor}
\usepackage{enumerate}
\usepackage{hyperref}
\hypersetup{
    colorlinks=true,       % false: boxed links; true: colored links
    linkcolor=blue,          % color of internal links
    citecolor=blue,        % color of links to bibliography
    filecolor=blue,      % color of file links
    urlcolor=blue           % color of external links
}

    \newtheorem{Lem}{Lemma}[section]
    \newtheorem{Lem-Def}[Lem]{Lemma-Definition}
    \newtheorem{Prop}[Lem]{Proposition}

    \newtheorem{Thm}[Lem]{Theorem}  
    \newtheorem{Cor}[Lem]{Corollary}

\theoremstyle{definition}
\font\smallsc=cmcsc10
\font\smallsl=cmsl10

    \newtheorem{Rem}[Lem]{Remark}

\newcommand{\col}{\colon}

\renewcommand{\l}{\ell}

\newcommand{\Aut}{\text{Aut}}

\newcommand{\val}{\text{val}}

\begin{document}

\title{The automorphism group of $M_{0,n}^\textnormal{trop}$ and $\overline{M}_{0,n}^{\textnormal{trop}}$}

\author{Alex Abreu and Marco Pacini}

\thanks{The second author was partially supported by CNPq, processo 304044/2013-0.}

\begin{abstract}
In this paper we show that the automorphism groups of $M_{0,n}^{\text{trop}}$ and $\overline{M}_{0,n}^{\text{trop}}$ are isomorphic to the permutation group $S_n$ for $n\geq5$, while the automorphism groups of $M_{0,4}^{\text{trop}}$ and  $\overline{M}_{0,4}^{\text{trop}}$ are isomorphic to the permutation group $S_3$.
\end{abstract}

 \maketitle

%\begin{itemize}
%\item detalhe prova prop prop:leaf
%\end{itemize}

\section{Introduction}

   The study of the biregular and birational geometry of $\overline{M}_{g,n}$, the moduli space of Deligne-Mumford stable curves, has recently attracted a lot of interest. Some natural issues, such as, for example, the computation of the automorphims group of $\overline{M}_{g,n}$ have been answered only in the last few years. In a series of papers (\cite{BM} and \cite{Ma}), Bruno, Massarenti and Mella proved that the automorphism group of $\overline{M}_{g,n}$ is the permutation group $S_n$, except in a few cases. This paper is devoted to the computation of the automorphism group of other moduli spaces which have an interesting geometric connection with $\overline{M}_{g,n}$.\par
	In the last decade, many interesting parallels have been made between tropical and algebraic geometry. Some classical and new results in algebraic geometry have been proven by means of tropical geometry, see for example \cite{CDPR}, \cite{JP} and \cite{JP1}. The tropical counterparts of $\overline{M}_{g,n}$ are the moduli space $M_{g,n}^{\text{trop}}$ of pointed tropical curves constructed in \cite{M}, \cite{Caporaso1} and \cite{BMV}, and its compactification $\overline{M}_{g,n}^{\text{trop}}$ constructed in \cite{Caporaso} by means of extended tropical curves. In \cite{ACP} the authors exhibited a geometrically meaningful connection between $\overline{M}_{g,n}$ and $\overline{M}_{g,n}^{\text{trop}}$.\par
	This paper is motivated by the following questions:
	\begin{enumerate}
	\item what are the automorphism group of $M_{g,n}^{\text{trop}}$ and $\overline{M}_{g,n}^{\text{trop}}$?
	\item what is the interplay between the automorphism group of $\overline{M}_{g,n}$ and the ones of $M_{g,n}^{\text{trop}}$ and $\overline{M}_{g,n}^{\text{trop}}$?
	\end{enumerate}
	We prove that the two groups in Question (1) are equal for $g=0$. Moreover for $n\geq 5$ the automorphism group of $M_{0,n}^{\text{trop}}$ is the symmetric group $S_n$, and the automorphism group of $M_{0,4}^{\text{trop}}$ is $S_3$. Our arguments strongly use that the graph underlying the considered tropical curves is a tree. So we do not see any trivial way to extend our techniques for higher $g$. At the end of the paper we show that the automorphism group of $M_2^{\text{trop}}$ is trivial and we briefly discuss Question (2).

\section{Preliminaries}

   A \emph{tree} is a connected graph without cycles. For a tree $\Gamma$, we denote by $V(\Gamma)$ and $E(\Gamma)$ its sets of vertices and edges, respectively. For a vertex $v\in V(\Gamma)$ we denote by $E(v)$ the set of edges incident to $v$, and by $\val(v)$ the cardinality of $E(v)$.   An \emph{isomorphism} $g$ between trees $\Gamma$ and $\Gamma'$ is defined as the data of bijections $g_V\col V(\Gamma)\to V(\Gamma')$ and $g_E\col E(\Gamma)\to E(\Gamma')$ which are compatible with incidence.
\begin{Rem}
\label{rem:iso}
If $g_1$ and $g_2$ are isomorphisms between trees $\Gamma$ and $\Gamma'$ with at least $3$ vertices such that $g_{1,E}=g_{2,E}$, then $g_1=g_2$.
\end{Rem}

	A \emph{legged tree} with legs indexed by the finite set $L$ (the set of legs) is the data of a tree $\Gamma$ and a map ${leg}_\Gamma\col L\to V(\Gamma)$. Usually, we will still write $\Gamma$ for a legged tree and denote by $L(\Gamma)$ its set of legs. Moreover we denote by $L(v)$ the set of legs incident to $v$, i.e., $L(v):={leg}_\Gamma^{-1}(v)$ and by $\l(v)$ the cardinality of $L(v)$. A \emph{$n$-legged tree} is a legged tree $\Gamma$ such that $L(\Gamma)=I_n:=\{1,\ldots,n\}$. A $n$-legged tree $\Gamma$ is \emph{stable} if $\val(v)+\l(v)\geq3$ for every $v\in V(\Gamma)$. A \emph{leaf} of a tree $\Gamma$ is a vertex with $\val(v)=1$. A \emph{chain} is a tree with only $2$ leaves. A \emph{path} in a tree $\Gamma$ is a subtree of $\Gamma$ that is a chain.\par
	 Given a subset $S\subset E(\Gamma)$, we define the legged tree $\Gamma/S$ as the legged tree obtained by contracting all edges in $S$. We say that a legged tree $\Gamma$ specializes to a legged tree $\Gamma'$ if there exists $S\subset E(\Gamma)$ such that $\Gamma'=\Gamma/S$.\par
	 An \emph{isomorphism} between $n$-legged trees $\Gamma$ and $\Gamma'$ is an isomorphism $g$ between the underlying trees which is also compatible with incidence of legs, i.e., $L(g_V(v))=L(v)$ for every $v\in V(\Gamma)$. We usually write $\Gamma\simeq\Gamma'$ if there exists an isomorphism between them.	Given a $n$-legged tree $\Gamma$ and a permutation $\sigma$ of $I_n$, we define the $n$-legged tree $\sigma(\Gamma)$ as the $n$-legged tree $\Gamma'$ with same underlying tree, but with $leg_{\Gamma'}=leg_{\Gamma}\circ \sigma^{-1}$. Given two $n$-legged trees $\Gamma$ and $\Gamma'$, and a permutation $\sigma$ of $I_n$ such that $\sigma(\Gamma)\simeq\Gamma'$, i.e., such that there exists an isomorphism $g$ between $\sigma(\Gamma)$ and $\Gamma'$, we denote by $\sigma(v)=g_V(v)$ and $\sigma(e)=g_E(e)$ for every $v\in V(\Gamma)$ and $e\in E(\Gamma)$.\par
   
\begin{Prop}
\label{prop:aut}
 A stable $n$-legged tree does not have nontrivial automorphisms.
\end{Prop}
\begin{proof}
An automorphism of a stable $n$-legged tree fixes the legs. Therefore each leaf of the graph must be fixed by such an automorphism, because each leaf must have at least one leg attached to it. Consider the stable legged tree obtained by removing a leaf and making the only edge incident to that leaf a leg. Since by induction on the number of vertices this graph has no nontrivial automorphisms, the result follows.
\end{proof}

	 A \emph{$n$-pointed tropical curve of genus 0} is the data of a $n$-legged tree $\Gamma$ together with a length function $E(\Gamma)\to \mathbb{R}_{>0}$. For a $n$-legged tree $\Gamma$, define the rational open polyhedral cone $C(\Gamma):=\mathbb{R}_{> 0}^{|E(\Gamma)|}$, and let  $\overline{C(\Gamma)}:=\mathbb{R}_{\geq 0}^{|E(\Gamma)|}$ be its closure.	The moduli space $M_{0,n}^{\text{trop}}$ of stable $n$-pointed tropical curves of genus 0 is the cone complex with cells $C(\Gamma)$, where $\Gamma$ runs through all stable $n$-legged trees, with glueing conditions specified by specializations. More precisely if $\Gamma$ specialized to $\Gamma'$, then $\overline{C(\Gamma')}$ is a face of $\overline{C(\Gamma)}$. The moduli space $\overline{M}_{0,n}^{\textnormal{trop}}$ of stable extended $n$-pointed tropical curves of genus $0$ is an extended cone complex which compactifies $M_{0,n}^{\textnormal{trop}}$. For more details about the terminology and the constructions of $M_{0,n}^{\text{trop}}$ and $\overline{M}_{0,n}^{\textnormal{trop}}$, see \cite[Section 2]{M}, \cite[Section 3]{Caporaso1}, \cite[Sections 2.1 and 3.2]{BMV} and \cite[Sections 2 and 4]{ACP}.\par
	 
	  An automorphism $f$ of $M_{0,n}^{\text{trop}}$ is a map of cone complexes $M_{0,n}^{\text{trop}}\to M_{0,n}^{\text{trop}}$ admitting an inverse which is also a map of cone complexes (see \cite[Section 2]{ACP}). Clearly an automorphism $f$ induces a permutation of the set of cells of $M_{0,n}^{\text{trop}}$ that preserves the dimension of each cell. \par
	Fix an automorphism $f$ of $M_{0,n}^{\text{trop}}$. Assume that $f(C(\Gamma))=C(\Gamma')$. By definition, $f|_{C(\Gamma)}$ is induced by an integral linear isomorphism $T\col\mathbb{R}^{|E(\Gamma)|}\to \mathbb{R}^{|E(\Gamma')|}$. Since $T(C(\Gamma))=C(\Gamma')$, we must have that $T$ sends the extremal rays of $\overline{C(\Gamma)}$ into the extremal rays of $\overline{C(\Gamma')}$, and since $T$ is primitive (because the inverse of $f|_{C(\Gamma)}$ must be integral as well),  it follows that $T$ is a permutation matrix, i.e., it is induced by a bijection between the sets $E(\Gamma)$ and $E(\Gamma')$. Abusing notation, we denote by $f\col E(\Gamma)\to E(\Gamma')$ such a bijection.  Note that given a subset $S\subset E(\Gamma)$, we have $f(C(\Gamma/S))=C(\Gamma'/f(S))$, because $\overline{C(\Gamma/S)}$ is a face of $\overline{C(\Gamma)}$.
	  
		An automorphism $\overline{f}$ of $\overline{M}_{0,n}^{\textnormal{trop}}$ is a morphism of extended cone complexes $\overline{M}_{0,n}^{\textnormal{trop}}\to\overline{M}_{0,n}^{\textnormal{trop}}$ admitting an inverse which is also a morphism of extended cone complexes (see \cite[Section 2]{ACP}).
		
		\begin{Prop}
		\label{prop:barra}
		There is a canonical isomorphism between the automorphism groups of $\overline{M}_{0,n}^{\textnormal{trop}}$ and $M_{0,n}^{\textnormal{trop}}$.
		\end{Prop}
		
		\begin{proof} We note that each automorphism of $M_{0,n}^{\textnormal{trop}}$ extends to an automorphism of $\overline{M}_{0,n}^{\textnormal{trop}}$ by linearity. Let $\overline{f}$ be an automorphism of $\overline{M}_{0,n}^{\textnormal{trop}}$. By the definition of morphism of extended cone complex (see \cite[Section 2]{ACP}), for each extended cone $\overline{\sigma}$ in $\overline{M}_{0,n}^{\textnormal{trop}}$, there exists an extended cone $\overline{\sigma}'$ in $\overline{M}_{0,n}^{\textnormal{trop}}$ such that $\overline{f}|_{\overline{\sigma}}$ factors through a morphims of extended cones $\overline{f}|_{\overline{\sigma}}\col\overline{\sigma}\to\overline{\sigma}'$. If $\overline{\sigma}$ is a maximal cone, then, since $\overline{f}$ is an automorphism, so must be $\overline{\sigma}'$. By definition of a morphism of extended cones (see \cite[Section 2]{ACP}), we have that $\overline{f}|_{\sigma}$ is a morphism of cones $\overline{f}|_{\sigma}\col\sigma\to \tau$ where $\tau$ is a face (possibly at infinity) of $\overline{\sigma}'$. Hence, since $\overline{f}|_\sigma$ is injective (because $\overline{f}$ is an automorphism) and $\dim(\sigma)=\dim(\sigma')$, we must have that $\tau=\sigma'$. The same holds true for every maximal extended cone of $\overline{M}_{0,n}^{\textnormal{trop}}$. Hence $\overline{f}(M_{0,n}^\text{trop})=M_{0,n}^{\textnormal{trop}}$ because $M_{0,n}^{\textnormal{trop}}$ is the union of its maximal cones. So the restriction of $\overline{f}$ to $M_{0,n}^{\textnormal{trop}}$ is an automorphism of $M_{0,n}^{\textnormal{trop}}$.\end{proof}

\section{The result}

Throughout the section $f$ will be a fixed automorphism of $M_{0,n}^{\text{trop}}$.

\begin{Prop}
\label{prop:count}
Let $\Gamma$ be a stable $n$-legged tree with $m$ edges. Then the number of $(m+1)$-dimensional cones in $M_{0,n}^{trop}$ whose closure contain $C(\Gamma)$ is 
\[
\sum_{v\in V(\Gamma)} \left(2^{\l(v)+\val(v)-1}-(\l(v)+\val(v)+1)\right).
\]
\end{Prop}
\begin{proof}
The number of $(m+1)$-dimensional cones whose closure contain $C(\Gamma)$ is precisely the number of stable $n$-legged trees with $m+1$ edges that specialize to $\Gamma$. To construct a stable $n$-legged tree $\Gamma'$ with $m+1$ edges that specializes to $\Gamma$, it is equivalent to replace a vertex $v$ of $\Gamma$ by two vertices $v_1$ and $v_2$, connected by an edge $e$, with $L(v)=L(v_1)\coprod L(v_2)$ and $E(v)=(E(v_1)\setminus\{e\})\coprod(E(v_2)\setminus\{e\})$.\par
   Since $\l(v_i)+\val(v_i)\geq 3$ for $i=1,2$, we have to make a partition of $L(v)\cup E(v)$ into two subsets such that each one has at least $2$ elements. Clearly, the number of ways to do this is $2^{\l(v)+\val(v)}-2(\l(v)+\val(v)+1)$. By symmetry, we have to divide by 2, and we obtain the result.
\end{proof}
For the next several results we note that there exists a permutation $\sigma$ of $I_n$ such that $\Gamma\simeq\sigma(\Gamma')$, if and only if $\Gamma$ and $\Gamma'$ are isomorphic as unlegged trees and the corresponding vertices have the same number of legs.\par

\begin{Cor}
\label{cor:2f}
Let $\Gamma$ and $\Gamma'$ be two stable $n$-legged trees with 2 vertices such that $f(C(\Gamma))=C(\Gamma')$. Then there exists a permutation $\sigma$ of $I_n$ such that $\sigma(\Gamma)\simeq\Gamma'$.
\end{Cor}
\begin{proof}
Let $v_1$, $v_2$ be the vertices of $\Gamma$ and $v_1'$, $v_2'$ the ones of $\Gamma'$. Since $f$ is an automorphism, the numbers of $2$-dimensional cones whose closure contain $C(\Gamma)$ and $C(\Gamma')$ are equal. Then we have 

\[
2^{\l(v_1)}+2^{\l(v_2)}-(\l(v_1)+\l(v_2)+4)=2^{\l(v_1')}+2^{\l(v_2')}-(\l(v_1')+\l(v_2')+4),
\]
and hence
\[
2^{\l(v_1)}+2^{\l(v_2)}=2^{\l(v_1')}+2^{\l(v_2')}
\]
from which we get $\l(v_1)=\l(v_1')$ or $\l(v_1)=\l(v_2')$ and the result follows.
\end{proof}
\begin{Lem}
\label{lem:count}
Let $a_i, b_i$, for $i=1,2,3$ be natural numbers such that $a_1+a_2+a_3=b_1+b_2+b_3$. If 
\[
2^{a_1}+2^{a_2}+2^{a_3}=2^{b_1}+2^{b_2}+2^{b_3}
\]
then $a_i=b_{\tau(i)}$ for some permutation $\tau$ of $\{1,2,3\}$.
\end{Lem}
\begin{proof}
Assume, without loss of generality, that $a_1\leq a_2\leq a_3$, $b_1\leq b_2\leq b_3$ and $a_3\geq b_3$. Clearly $a_3\leq b_3+1$, otherwise 
\[
2^{a_3}> 3\cdot 2^{b_3}\geq 2^{b_1}+2^{b_2}+2^{b_3},
\]
a contradiction. Therefore, either $a_3=b_3$ or $a_3=b_3+1$. In the former case the result follows trivially. In the latter, we have
\[
2^{a_1}+2^{a_2}+2^{b_3}=2^{b_1}+2^{b_2}
\]
which implies that $b_3=b_2$, hence $2^{b_1}=2^{a_1}+2^{a_2}$ and therefore $b_1=a_1+1=a_2+1$. Now, we have $a_1+a_1+b_2+1=a_1+1+b_2+b_2$, which implies that $a_1=b_2$, which contradicts the fact that $b_2\geq b_1$.
\end{proof}

\begin{Prop}
\label{prop:3f}
Let $\Gamma$ and $\Gamma'$ be two stable $n$-legged trees with 3 vertices such that $f(C(\Gamma))=C(\Gamma')$. Then there exists a permutation  $\sigma$ of $I_n$ such that $\sigma(\Gamma)\simeq\Gamma'$ and $f(e)=\sigma(e)$ for every $e\in E(\Gamma)$.
\end{Prop}
\begin{proof}
Let $v_1$, $v_2$, $v_3$ be the vertices of $\Gamma$ and $v_1'$, $v_2'$, $v_3'$ the ones of $\Gamma'$, where $v_2$ and $v_2'$ are not leaves. Let $e_i$ be the edge between $v_i$ and $v_{i+1}$ and $e_i'$ be the one between $v_i'$ and $v_{i+1}'$, for $i=1,2$. By the same argument in the proof of Corollary \ref{cor:2f} we get that 
\[
2^{\l(v_1)}+2^{\l(v_2)+1}+2^{\l(v_3)}=2^{\l(v_1')}+2^{\l(v_2')+1}+2^{\l(v_3')}.
\]
Using Lemma \ref{lem:count} we conclude that, up to relabeling $v_1$ and $v_3$, either $\l(v_i)=\l(v_i')$ for $i=1,2,3$ or 
\begin{equation}
\label{eq:L}
\l(v_1)=\l(v_2')+1,\quad \l(v_2)=\l(v_1')-1, \quad \l(v_3)=\l(v_3').
\end{equation} \par
In the former case, we just have to prove that $f(e_1)=e_1'$. If $f(e_1)=e_2'$, then contracting $e_1$ and $f(e_1)=e_2'$ we get that $f(C(\Gamma/\{e_1\}))=C(\Gamma'/\{e_2'\})$ and by Corollary \ref{cor:2f} we get that either $\l(v_3)=\l(v_1')$ or $\l(v_3)=\l(v_2')+\l(v_3')$. The second equality can not occur because $\l(v_3)=\l(v_3')$ and $\l(v_2')>0$, hence $\l(v_3)=\l(v_1')=\l(v_1)$ and then switching $v_1'$ and $v_3'$ we get the result.\par
   In the latter case we argue as follows.  We have two cases. In the first case $f(e_1)=e_1'$. Contracting $e_2$ and $f(e_2)=e_2'$, we get that $f(C(\Gamma/\{e_2\}))=C(\Gamma'/\{e_2'\})$ and by Corollary \ref{cor:2f} we get that either $\l(v_1)=\l(v_1')$ or $\l(v_1)=\l(v_2')+\l(v_3')$. The former, together with Equation \eqref{eq:L} implies that $\l(v_2)=\l(v_2')$ and the result follows, while the latter is clearly impossible because $\l(v_3')\geq 2$.	In the second case $f(e_1)=e_2'$. Using Corollary \ref{cor:2f} for both $f(C(\Gamma/\{e_2\}))=C(\Gamma'/\{e_1'\})$ and $f(C(\Gamma/\{e_1\}))=C(\Gamma'/\{e_2'\})$, we get that (excluding the obvious impossible cases) $\l(v_1)=\l(v_3')$ and $\l(v_3)=\l(v_1')$. This, together with Equation \eqref{eq:L}, implies that $\l(v_i)=\l(v_i')$ for $i=1,2,3$ and $\l(v_1)=\l(v_3)$. In particular, switching $v_1'$ and $v_3'$ we get the result.
\end{proof}

\begin{Prop}
\label{prop:leaf}
Let $\Gamma$ and $\Gamma'$ be two stable $n$-legged trees such that $f(C(\Gamma))=C(\Gamma')$. Then there exist leaves $v$, $v'$ of $\Gamma$, $\Gamma'$ such that $\l(v)=\l(v')$ and $f(E(v))=E(v')$. 
\end{Prop}

\begin{proof}
 Let $v$ be a leaf of $\Gamma$ and $e$ be the only edge attached to $v$. Contracting all edges of $\Gamma$ except $e$, we get a stable $n$-legged tree with 2 vertices $v$ and $\overline{v}$. Contracting all edges of $\Gamma'$ except $f(e)$, by Corollary \ref{cor:2f} we must get a stable $n$-legged tree with $2$-vertices $w'$ and $\overline{w}'$, such that $\l(v)=\l(w')$ (up to switching $w'$ and $\overline{w}'$). If the vertex of $f(e)$ that contracts to $w'$ is (respectively, is not) a leaf, then there exists a leaf $v'$ of $\Gamma'$ with $\l(v')=\l(w')$ (respectively, $\l(v')<\l(w')$). In particular, by the same argument applied to $f^{-1}$, for any given leaf $v'$ of $\Gamma'$, there exists a leaf $w$ in $\Gamma$ such that $\l(w)\leq \l(v')$.\par
    In the setting above, if we choose $v$ to be a leaf of $\Gamma$ with the minimum number of legs attached to it, then it follows that the vertex $v'$ attached to $f(e)$ that contracts to $w'$ is also a leaf, otherwise there would be a leaf $w$ of $\Gamma$ such that $\l(w)<\l(v)$, a contradiction. Note that we also get $\l(v)=\l(v')$.
\end{proof}

\begin{Prop}
\label{prop:chain}
Let $\Gamma$ and $\Gamma'$ be two stable $n$-legged chains with 4 vertices such that $f(C(\Gamma))=C(\Gamma')$. Then there exists a permutation $\sigma$ of $I_n$ such that $\sigma(\Gamma)\simeq\Gamma'$ and $f(e)=\sigma(e)$ for every $e\in E(\Gamma)$. In particular $f$ takes the edge of $\Gamma$ attached to no leaf to the edge of $\Gamma'$ attached to no leaf.
\end{Prop}

\begin{proof}
Assume that $v_i$ (respectively, $v_i'$) are the vertices of $\Gamma$ (respectively, $\Gamma'$) for $i=1,2,3,4$, and $e_i$ (respectively, $e_i'$) the edge connecting $v_i$ and $v_{i+1}$ (respectively, $v_i'$ and $v_{i+1}'$), for $i=1,2,3$. By Proposition \ref{prop:leaf} we can assume, without loss of generality, that $\l(v_1)=\l(v_1')$ and $f(e_1)=e_1'$. Then we have two cases.\par
 In the first case $f(e_2)=e_2'$ and $f(e_3)=e_3'$, hence contracting $e_1$ and $f(e_1)=e_1'$, and applying Proposition \ref{prop:3f}, we get that $\l(v_3)=\l(v_3')$ and $\l(v_4)=\l(v_4')$ hence $\l(v_2)=\l(v_2')$ and the result follows.\par
  In the second case $f(e_2)=e_3'$ and $f(e_3)=e_2'$. Contracting $e_1$ and $f(e_1)=e_1'$, and applying Proposition \ref{prop:3f}, we get that $\l(v_3)=\l(v_3')$, $\l(v_4)=\l(v_1')+\l(v_2')$ and $\l(v_4')=\l(v_1)+\l(v_2)$. These equalities translates to $\l(v_1')=\l(v_1)$, $\l(v_2')=\l(v_4)-\l(v_1)$, $\l(v_3')=\l(v_3)$ and $\l(v_4')=\l(v_1)+\l(v_2)$. Now contracting $e_2$ and $f(e_2)=e_3'$, we get that $\l(v_2)+\l(v_3)=\l(v_2')$, which implies $\l(v_2)+\l(v_3)=\l(v_4)-\l(v_1)$. Contracting $e_3$ and $f(e_3)=e_2'$, we get that $\l(v_2)=\l(v_2')+\l(v_3')$, which implies that $\l(v_2)=\l(v_4)-\l(v_1)+\l(v_3)$. Clearly this yields $\l(v_3)=0$, contradiction.
\end{proof}

\begin{Prop}
\label{prop:sigma}
Let $\Gamma$ and $\Gamma'$ be two stable $n$-legged trees such that $f(C(\Gamma))=C(\Gamma')$. Then there exists a permutation $\sigma$ of $I_n$ such that $\sigma(\Gamma)\simeq\Gamma'$ and $f(e)=\sigma(e)$ for every $e\in E(\Gamma)$.
\end{Prop}
\begin{proof}
		The proof is by induction on the number of edges of $\Gamma$. If $\Gamma$ has one or two edges, then the result is Corollary \ref{cor:2f} and Proposition \ref{prop:3f}. Assume now that $\Gamma$ has at least $3$ edges. By Proposition \ref{prop:leaf} there exist leaves $v_1$ and $v_1'$ of $\Gamma$ and $\Gamma'$ such that $\l(v_1)=\l(v_1')$ and $f(e_1)=e_1'$ where $e_1$ and $e_1'$ are the unique edges attached to $v_1$ and $v_1'$. By the induction hypothesis, upon contracting $e_1$ and $f(e_1)=e_1'$ we get that there exists $\sigma$ such that $\sigma(\Gamma/\{e\})\simeq\Gamma'/\{e'\}$ and $f(e)=\sigma(e)$ for every $e\in E(\Gamma/\{e_1\})=E(\Gamma)\setminus\{e_1\}$. Let $v_2$ and $v_2'$be  the other vertices connected to $e_1$ and $e_1'$. Since $\Gamma$ is a tree, there exists a unique path from $v_2$ to $\sigma^{-1}(v_2')$.  If $\sigma(v_2)=v_2'$, then the result follows. Otherwise, we have 2 cases.\par

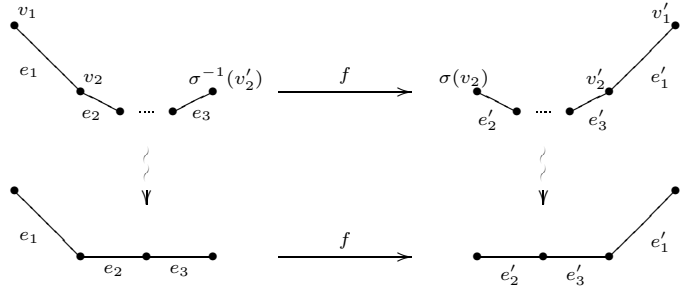
\begin{figure}[h]
\[
\begin{xy} <25pt,0pt>:
(0,0)*{\scriptstyle\bullet}="b"; 
(-1,1)*{\scriptstyle\bullet}="a";
(2,0)*{\scriptstyle\bullet}="e";
(0.6,-0.3)*{\scriptstyle\bullet}="c";
(1.4,-0.3)*{\scriptstyle\bullet}="d";
(6,0)*{\scriptstyle\bullet}="e1"; 
(9,1)*{\scriptstyle\bullet}="a1";
(8,0)*{\scriptstyle\bullet}="b1";
(6.6,-0.3)*{\scriptstyle\bullet}="d1";
(7.4,-0.3)*{\scriptstyle\bullet}="c1";
{\ar^{f}@{->}(3,0)*{};(5,0)*{}};
{\ar_{e_1}@{-}"a"*{};"b"*{}};
{\ar_{e_2}@{-}"b"*{};"c"*{}};
{\ar_{e_3}@{-}"d"*{};"e"*{}};
{\ar@{..}"c"+(0.3,0)*{};"d"+(-0.3,0)*{}};
{\ar^{e_1'}@{-}"a1"*{};"b1"*{}};
{\ar^{e_3'}@{-}"b1"*{};"c1"*{}};
{\ar^{e_2'}@{-}"d1"*{};"e1"*{}};
{\ar@{..}"c1"+(-0.3,0)*{};"d1"+(0.3,0)*{}};
"a"+(0.2,0.2)*{\scriptstyle{v_1}};
"b"+(0.2,0.2)*{\scriptstyle{v_2}};
"e"+(0.2,0.2)*{\scriptstyle{\sigma^{-1}(v_2')}};
"a1"+(-0.2,0.2)*{\scriptstyle{v_1'}};
"b1"+(-0.2,0.2)*{\scriptstyle{v_2'}};
"e1"+(-0.2,0.2)*{\scriptstyle{\sigma(v_2)}};
(0,-2.5)*{\scriptstyle\bullet}="b2"; 
(-1,-1.5)*{\scriptstyle\bullet}="a2";
(2,-2.5)*{\scriptstyle\bullet}="e2";
(1,-2.5)*{\scriptstyle\bullet}="c2";
(6,-2.5)*{\scriptstyle\bullet}="e3"; 
(9,-1.5)*{\scriptstyle\bullet}="a3";
(8,-2.5)*{\scriptstyle\bullet}="b3";
(7,-2.5)*{\scriptstyle\bullet}="c3";
{\ar^{f}@{->}(3,-2.5)*{};(5,-2.5)*{}};
{\ar_{e_1}@{-}"a2"*{};"b2"*{}};
{\ar_{e_2}@{-}"b2"*{};"c2"*{}};
{\ar_{e_3}@{-}"c2"*{};"e2"*{}};
{\ar^{e_1'}@{-}"a3"*{};"b3"*{}};
{\ar^{e_3'}@{-}"b3"*{};"c3"*{}};
{\ar^{e_2'}@{-}"c3"*{};"e3"*{}};
{\ar^{}@{~}(1,-0.7)*{};(1,-1.5)*{}};
{\ar@{->}(1,-1.5)*{};(1,-1.7)*{}};
{\ar^{}@{~}(7,-0.7)*{};(7,-1.5)*{}};
{\ar@{->}(7,-1.5)*{};(7,-1.7)*{}};
\end{xy}
\]
\caption{The first case.}
\end{figure}

		 In the first case the path has at least $2$ edges. Let $e_2$ and $e_3$ be the edges of the path attached respectively to $v_2$ and $\sigma^{-1}(v_2')$, and let $e_2':=f(e_2)$ and $e_3':=f(e_3)$. Contracting all edges of $\Gamma$, except for $e_1$, $e_2$, $e_3$ and all edges of $\Gamma'$, except for $e_1'$, $e_2'$, $e_3'$, we get two stable $n$-legged chains $K$ and $K'$ with 4 vertices such that $f(C(K))=C(K')$. The edges $e_2$ of $K$ and $e_3'$ of $K'$ are attached to no leaf. However Proposition \ref{prop:chain} applied to $K$ and $K'$ implies that $f(e_2)=e_3'$, contradiction.\par

		In the second case the path has only one edge $e_2$. Let $e_2':=f(e_2)$. We claim that, up to switching $f$ with $f^{-1}$ and $\Gamma$ with $\Gamma'$, we can assume that $v_2'$ has no other edge attached to it other than $e_1'$ and $e_2'$. Note that this also implies that $\sigma^{-1}(v_2')$ is a leaf. To prove the claim, contract the set $S$ of all edges that belong to the connected component of $\Gamma\setminus\{e_1,e_2\}$ that contains $v_2$. Let $\overline{v_2}$ and $\overline{v_2}'$ be the vertices of $\Gamma/S$ and $\Gamma'/f(S)$, respectively, to which $v_2$ and $v_2'$ contracts. If at least one edge was contracted, i.e., $S$ is nonempty, then, by the induction hypothesis, there exists a permutation $\sigma'$ of $I_n$, such that $\sigma'(\Gamma/S)\simeq\Gamma'/f(S)$ and $f(e)=\sigma'(e)$ for every $e\in E(\Gamma/S)=E(\Gamma)\setminus S$. This implies that $\sigma'(e_1)=f(e_1)=e_1'$, hence $\sigma'(\overline{v_2})=\overline{v_2}'$, and since $\overline{v_2}$ has no edges attached to it other than $e_1$ and $e_2$, the only edges attached to $\overline{v_2}'$ are $e_1'$ and $e_2'$. However, no edge attached to $v_2'$ was contracted, hence we have proved that $E(v_2')=\{e_1',e_2'\}$.	If $S=\emptyset$, then we get that the only edges attached to $v_2$ are $e_1$ and $e_2$, which proves our claim after switching $\Gamma$ with $\Gamma'$ and $f$ with $f^{-1}$.\par
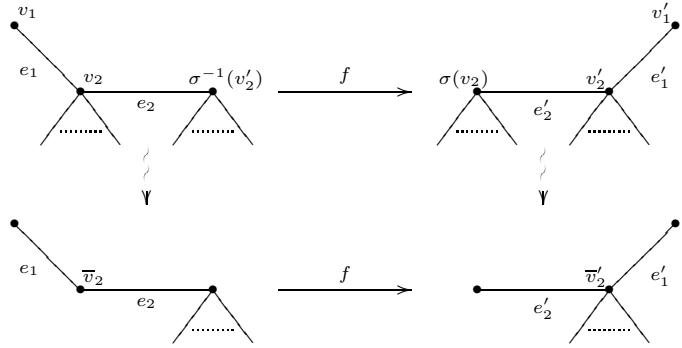
\begin{figure}[h]
\[
\begin{xy} <25pt,0pt>:
(0,0)*{\scriptstyle\bullet}="b"; 
(-1,1)*{\scriptstyle\bullet}="a";
(2,0)*{\scriptstyle\bullet}="c";
(6,0)*{\scriptstyle\bullet}="c1"; 
(9,1)*{\scriptstyle\bullet}="a1";
(8,0)*{\scriptstyle\bullet}="b1";
{\ar^{f}@{->}(3,0)*{};(5,0)*{}};
{\ar_{e_1}@{-}"a"*{};"b"*{}};
{\ar_{e_2}@{-}"b"*{};"c"*{}};
{\ar^{e_1'}@{-}"a1"*{};"b1"*{}};
{\ar^{e_2'}@{-}"b1"*{};"c1"*{}};
"a"+(0.2,0.2)*{\scriptstyle{v_1}};
"b"+(0.2,0.2)*{\scriptstyle{v_2}};
"c"+(0.2,0.2)*{\scriptstyle{\sigma^{-1}(v_2')}};
"a1"+(-0.2,0.2)*{\scriptstyle{v_1'}};
"b1"+(-0.2,0.2)*{\scriptstyle{v_2'}};
"c1"+(-0.2,0.2)*{\scriptstyle{\sigma(v_2)}};
(0,-3)*{\scriptstyle\bullet}="b2"; 
(-1,-2)*{\scriptstyle\bullet}="a2";
(2,-3)*{\scriptstyle\bullet}="c2";
(6,-3)*{\scriptstyle\bullet}="c3"; 
(9,-2)*{\scriptstyle\bullet}="a3";
(8,-3)*{\scriptstyle\bullet}="b3";
"b2"+(0.2,0.2)*{\scriptstyle{\overline{v}_2}};
"b3"+(-0.2,0.2)*{\scriptstyle{\overline{v}_2'}};
{\ar^{f}@{->}(3,-3)*{};(5,-3)*{}};
{\ar_{e_1}@{-}"a2"*{};"b2"*{}};
{\ar_{e_2}@{-}"b2"*{};"c2"*{}};
{\ar^{e_1'}@{-}"a3"*{};"b3"*{}};
{\ar^{e_2'}@{-}"b3"*{};"c3"*{}};
{\ar^{}@{~}(1,-0.7)*{};(1,-1.5)*{}};
{\ar@{->}(1,-1.5)*{};(1,-1.7)*{}};
{\ar^{}@{~}(7,-0.7)*{};(7,-1.5)*{}};
{\ar@{->}(7,-1.5)*{};(7,-1.7)*{}};
{\ar@{-}"b"*{};"b"+(-0.6,-0.8)*{}};
{\ar@{-}"b"*{};"b"+(0.6,-0.8)*{}};
{\ar@{..}"b"+(-0.3,-0.6)*{};"b"+(0.3,-0.6)*{}};
{\ar@{-}"c1"*{};"c1"+(-0.6,-0.8)*{}};
{\ar@{-}"c1"*{};"c1"+(0.6,-0.8)*{}};
{\ar@{..}"c1"+(-0.3,-0.6)*{};"c1"+(0.3,-0.6)*{}};
{\ar@{-}"c"*{};"c"+(-0.6,-0.8)*{}};
{\ar@{-}"c"*{};"c"+(0.6,-0.8)*{}};
{\ar@{..}"c"+(-0.3,-0.6)*{};"c"+(0.3,-0.6)*{}};
{\ar@{-}"b1"*{};"b1"+(-0.6,-0.8)*{}};
{\ar@{-}"b1"*{};"b1"+(0.6,-0.8)*{}};
{\ar@{..}"b1"+(-0.3,-0.6)*{};"b1"+(0.3,-0.6)*{}};
{\ar@{-}"c2"*{};"c2"+(-0.6,-0.8)*{}};
{\ar@{-}"c2"*{};"c2"+(0.6,-0.8)*{}};
{\ar@{..}"c2"+(-0.3,-0.6)*{};"c2"+(0.3,-0.6)*{}};
{\ar@{-}"b3"*{};"b3"+(-0.6,-0.8)*{}};
{\ar@{-}"b3"*{};"b3"+(0.6,-0.8)*{}};
{\ar@{..}"b3"+(-0.3,-0.6)*{};"b3"+(0.3,-0.6)*{}};
\end{xy}
\]
\caption{The second case: first contraction}
\end{figure}

		Let us come back to the proof of the second case. Contracting all edges of $\Gamma$ except $e_1$ and $e_2$ and all edges of $\Gamma'$ except $e_1'$ and $e_2'$,  we end up in the conditions of Proposition \ref{prop:3f}. Denote by $\overline{v}_1'$, $\overline{v}_2'$ and $\overline{\sigma^{-1}(v_2')}$ the vertices to which $v_1'$, $v_2'$ and $\sigma^{-1}(v_2')$ contract. Since $v_1'$, $v_2'$ and $\sigma^{-1}(v_2')$  have no edges attached to them that are contracted, we get that $L(\overline{v_1}')=L(v_1')$, $L(\overline{v_2}')=L(v_2')$ and $L(\overline{\sigma^{-1}(v_2')})=L(\sigma^{-1}(v_2'))$. Moreover, since $f(e_2)=e_2'$, we have that
\[
\l(\sigma^{-1}(v_2'))=n-\l(v_2')-\l(v_1').
\]

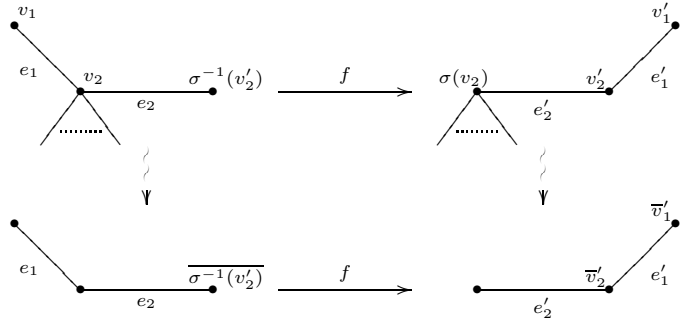
\begin{figure}[h]
\[
\begin{xy} <25pt,0pt>:
(0,0)*{\scriptstyle\bullet}="b"; 
(-1,1)*{\scriptstyle\bullet}="a";
(2,0)*{\scriptstyle\bullet}="c";
(6,0)*{\scriptstyle\bullet}="c1"; 
(9,1)*{\scriptstyle\bullet}="a1";
(8,0)*{\scriptstyle\bullet}="b1";
{\ar^{f}@{->}(3,0)*{};(5,0)*{}};
{\ar_{e_1}@{-}"a"*{};"b"*{}};
{\ar_{e_2}@{-}"b"*{};"c"*{}};
{\ar^{e_1'}@{-}"a1"*{};"b1"*{}};
{\ar^{e_2'}@{-}"b1"*{};"c1"*{}};
"a"+(0.2,0.2)*{\scriptstyle{v_1}};
"b"+(0.2,0.2)*{\scriptstyle{v_2}};
"c"+(0.2,0.2)*{\scriptstyle{\sigma^{-1}(v_2')}};
"a1"+(-0.2,0.2)*{\scriptstyle{v_1'}};
"b1"+(-0.2,0.2)*{\scriptstyle{v_2'}};
"c1"+(-0.2,0.2)*{\scriptstyle{\sigma(v_2)}};
(0,-3)*{\scriptstyle\bullet}="b2"; 
(-1,-2)*{\scriptstyle\bullet}="a2";
(2,-3)*{\scriptstyle\bullet}="c2";
(6,-3)*{\scriptstyle\bullet}="c3"; 
(9,-2)*{\scriptstyle\bullet}="a3";
(8,-3)*{\scriptstyle\bullet}="b3";
"a3"+(-0.2,0.2)*{\scriptstyle{\overline{v}_1'}};
"c2"+(0.2,0.2)*{\scriptstyle{\overline{\sigma^{-1}(v_2')}}};
"b3"+(-0.2,0.2)*{\scriptstyle{\overline{v}_2'}};
{\ar^{f}@{->}(3,-3)*{};(5,-3)*{}};
{\ar_{e_1}@{-}"a2"*{};"b2"*{}};
{\ar_{e_2}@{-}"b2"*{};"c2"*{}};
{\ar^{e_1'}@{-}"a3"*{};"b3"*{}};
{\ar^{e_2'}@{-}"b3"*{};"c3"*{}};
{\ar^{}@{~}(1,-0.7)*{};(1,-1.5)*{}};
{\ar@{->}(1,-1.5)*{};(1,-1.7)*{}};
{\ar^{}@{~}(7,-0.7)*{};(7,-1.5)*{}};
{\ar@{->}(7,-1.5)*{};(7,-1.7)*{}};
{\ar@{-}"b"*{};"b"+(-0.6,-0.8)*{}};
{\ar@{-}"b"*{};"b"+(0.6,-0.8)*{}};
{\ar@{..}"b"+(-0.3,-0.6)*{};"b"+(0.3,-0.6)*{}};
{\ar@{-}"c1"*{};"c1"+(-0.6,-0.8)*{}};
{\ar@{-}"c1"*{};"c1"+(0.6,-0.8)*{}};
{\ar@{..}"c1"+(-0.3,-0.6)*{};"c1"+(0.3,-0.6)*{}};
%{\ar@{-}"c"*{};"c"+(-0.6,-0.8)*{}};
%{\ar@{-}"c"*{};"c"+(0.6,-0.8)*{}};
%{\ar@{..}"c"+(-0.3,-0.6)*{};"c"+(0.3,-0.6)*{}};
%{\ar@{-}"b1"*{};"b1"+(-0.6,-0.8)*{}};
%{\ar@{-}"b1"*{};"b1"+(0.6,-0.8)*{}};
%{\ar@{..}"b1"+(-0.3,-0.6)*{};"b1"+(0.3,-0.6)*{}};
%{\ar@{-}"c2"*{};"c2"+(-0.6,-0.8)*{}};
%{\ar@{-}"c2"*{};"c2"+(0.6,-0.8)*{}};
%{\ar@{..}"c2"+(-0.3,-0.6)*{};"c2"+(0.3,-0.6)*{}};
%{\ar@{-}"b3"*{};"b3"+(-0.6,-0.8)*{}};
%{\ar@{-}"b3"*{};"b3"+(0.6,-0.8)*{}};
%{\ar@{..}"b3"+(-0.3,-0.6)*{};"b3"+(0.3,-0.6)*{}};
\end{xy}
\]
\caption{The second case: second contraction}
\end{figure}

Applying Proposition \ref{prop:count} to $\Gamma$ and $\Gamma'$, we must have that
\[
2^{\l(v_2)+\val(v_2)-1}+2^{\l(\sigma^{-1}(v_2'))}=2^{\l(\sigma(v_2))+\val(\sigma(v_2))-1}+2^{\l(v_2')+1}
\]
hence, since $\val(\sigma(v_2))=\val(v_2)-1$ and $\l(\sigma(v_2))=\l(v_2)$, we get $\l(v_2')+1=\l(v_2)+\val(v_2)-1$. However, we have that the number of legs in $\Gamma$ must be at least
\[
\l(v_1)+\l(v_2)+\l(\sigma^{-1}(v_2'))+2(\val(v_2)-2)
\]
and, using that $\l(\sigma^{-1}(v_2'))=\l(v_2')$, this implies that
\begin{align*}
n&\geq \l(v_1)+\l(v_2')+(\l(v_2)+\val(v_2)-2)+\val(v_2)-2\\
 &=\l(v_1)+2\l(v_2')+\val(v_2)-2\\
&=n+\val(v_2)-2.
\end{align*}
Hence we have $\val(v_2)=2$, which implies that $\Gamma$ has only 3 vertices, from which the result follows from Proposition \ref{prop:3f}.
\end{proof}

\begin{Rem}
\label{rem:spec}
  In this remark, for a $n$-legged tree $\Gamma$, we denote by $\widetilde{\Gamma}$ its underlying tree.\par
  When $f(C(\Gamma))=C(\Gamma')$, Proposition \ref{prop:sigma} and Remark \ref{rem:iso} shows that $f$ induces a unique isomorphism $g_\Gamma\col\widetilde{\Gamma}\to\widetilde{\Gamma'}$ unless $\Gamma$ has exactly 2 vertices with the same number of legs incident to them, in which case the two isomorphisms between $\widetilde{\Gamma}$ and $\widetilde{\Gamma'}$ can be induced by some permutation $\sigma$ that satisfies Proposition \ref{prop:sigma}. The unique isomorphisms induced by $f$ are compatible with specializations, namely, for every $S\subset E(\Gamma)$ the isomorphism $g_{\Gamma/S}\col\widetilde{\Gamma/S}\to\widetilde{\Gamma'/f(S)}$ is induced by $g_\Gamma$ via specialization. Moreover, in the case where $\Gamma$ has exactly $2$ vertices with the same number of legs incident to them, only one of the two possible isomorphisms is compatible with specializations, unless $n=4$. To prove such a claim just choose a $n$-legged tree with $3$ vertices that specializes to $\Gamma$. We will again abuse notation and denote $f(v)=g_\Gamma(v)$ for every $v\in V(\Gamma)$.
\end{Rem}

Let $A$ be a subset of $I_n$, with $2\leq |A|\leq n-2$. Define $\Gamma_A$ the $n$-legged tree with exactly $2$ vertices $v_A$ and $\overline{v}_A$ such that $L(v_A)=A$.
\begin{Prop}
\label{prop:cup}
If $B\subset A$, with $A, B\subset I_n$ with $2\leq |A|,|B|\leq n-2$, then $L(f(v_B))\subset L(f(v_A))$. 
\end{Prop}
\begin{proof}
If $B\neq A$, let $\Gamma$ be the $n$-legged tree with exactly $3$ vertices $w_1$, $w_2$ and $w_3$, with $w_1$ and $w_3$ being the leaves, such that $L(w_1)=B$ and $L(w_2)=A\setminus B$. By Proposition \ref{prop:sigma} we can write $L(f(w_1))=B'$ and $L(f(w_2))=C'$, with $|B'|=|B|$ and $|C'|=|A|-|B|$. Contracting the edge between $w_2$ and $w_3$ in $\Gamma$ we get $\Gamma_B$, hence, by Remark \ref{rem:spec}, we get $L(f(v_B))=B'$. Contracting the edge between $w_1$ and $w_2$ in $\Gamma$ we get $\Gamma_A$, hence, again by Remark \ref{rem:spec}, $L(f(v_A))=B'\cup C'$. Therefore we have $L(f(v_B))\subset L(f(v_A))$. 
\end{proof}

\begin{Cor}
\label{cor:cup}
Let $n\geq 5$ and $\Gamma_i$ be a stable $n$-legged trees with exactly 2 vertices $v_i$ and $\overline{v_i}$ for $i=1,2,3$, such that $\l(v_1)=\l(v_2)=2$, $|L(v_1)\cap L(v_2)|=1$ and $L(v_3)=L(v_1)\cup L(v_2)$. Then $|L(f(v_1))\cap L(f(v_2))|=1$ and $L(f(v_3))=L(f(v_1))\cup L(f(v_2))$.
\end{Cor}
\begin{proof}
 By Proposition \ref{prop:cup}, we have $L(f(v_1))\subset L(f(v_3))$ and $L(f(v_2))\subset L(f(v_3))$. Since $L(f(v_1))\neq L(f(v_2))$ (otherwise $f(C(\Gamma_1))=f(C(\Gamma_2))$) and $\l(f(v_i))=\l(v_i)$ for $i=1,2,3$  (by Corollary \ref{cor:2f}), the result follows.
\end{proof}

\begin{Prop}
\label{prop:sigma2}
There exists a permutation $\sigma$ of $I_n$ such that, for every stable $n$-legged tree $\Gamma$ with exactly $2$ vertices, we have $L(f(v))=\sigma(L(v))$ for all $v\in V(\Gamma)$ with $|L(v)|=2$.
\end{Prop}
\begin{proof}
 By Corollary \ref{cor:cup}, there exist distinct elements $i_1$, $i_2$ and $i_3$ of $I_n$ such that 
\[
L(f(v_{\{1,2\}})=\{i_1,i_2\},\;L(f(v_{\{1,3\}}))=\{i_1,i_3\},\; L(f(v_{\{1,2,3\}}))=\{i_1,i_2,i_3\}.
\] Then, by Proposition \ref{prop:cup}, we have that $L(f(v_{\{2,3\}}))\subset \{i_1,i_2,i_3\}$. In this way, since 
\[
L(f(v_{\{2,3\}}))\neq L(f(v_{\{1,2\}}))\quad\text{and}\quad L(f(v_{\{2,3\}}))\neq L(f(v_{\{1,3\}})),
\]
 we get $L(f(v_{\{2,3\}}))=\{i_2,i_3\}$. By Corollary \ref{cor:cup} we have 
\[
L(f(v_{\{1,4\}}))\cap L(f(v_{\{1,2\}}))\neq\emptyset\quad\text{and}\quad L(f(v_{\{1,4\}}))\cap L(f(v_{\{1,3\}}))\neq\emptyset.
\]
Thus, we get that either $i_1\in L(f(v_{\{1,4\}}))$ or $\{i_2,i_3\}=L(f(v_{\{1,4\}}))$. The latter case can not happen because $L(f(v_{\{1,4\}}))\neq L(f(v_{\{2,3\}}))$. Hence, there exists $i_4\in I_n\setminus\{i_1,i_2,i_3\}$ such that $L(f(v_{\{1,4\}}))=\{i_1,i_4\}$. Analogously, $i_2\in L(f(v_{\{2,4\}}))$, and since $L(f(v_{\{1,4\}}))\cap L(f(v_{\{2,4\}}))\neq\emptyset$, we get that $L(f(v_{\{2,4\}}))=\{i_2,i_4\}$. Iterating the argument, we find indices $i_j\in I_n$ such that $L(f(v_{\{j,k\}}))=\{i_j,i_k\}$ and we can define $\sigma(j):=i_j$.
\end{proof}
\begin{Cor}
\label{cor:sigma3}
There exists a permutation $\sigma$ of $I_n$ such that, for every stable $n$-legged tree $\Gamma$ with exactly $2$ vertices, we have $L(f(v))=\sigma(L(v))$ for all $v\in V(\Gamma)$.
\end{Cor}
\begin{proof}
Let $B\subset L(v)$ such that $|B|=2$. Then, by Proposition \ref{prop:sigma2}, there exists a permutation $\sigma$ of $I_n$, which does not depend on $B$, such that $\sigma(B)=L(f(v_B))$. By Proposition \ref{prop:cup}, we have that $\sigma(B)\subset L(f(v))$. Since this holds for all $B$ with $|B|=2$, then $L(f(v))=\sigma(L(v))$.
\end{proof}

\begin{Thm}
\label{thm:main}
For $n\geq5$ the automorphism groups of $M_{0,n}^{\textnormal{trop}}$ and $\overline{M}_{0,n}^{\textnormal{trop}}$ are isomorphic to $S_n$. The automorphism groups of $M_{0,4}^{\textnormal{trop}}$ and $\overline{M}_{0,4}^{\textnormal{trop}}$ are isomorphic to $S_3$.
\end{Thm}
\begin{proof}
Note that by Proposition \ref{prop:barra} it is enough to prove the theorem for $M_{0,n}^{\textnormal{trop}}$.\par
First of all, we show that, for $n\geq 5$, there is an injective group homomorphism from $S_n$ to the automorphism group of $M_{0,n}^{\text{trop}}$. Clearly, given a permutation $\sigma\in S_n$, there is an automorphism of $M_{0,n}^{\text{trop}}$ defined by the map sending $\Gamma$ to $\sigma(\Gamma)$ and preserving the lenghts of the edges. If $\sigma$ induces the identity in $M_{0,n}^{\text{trop}}$, then $\sigma(\Gamma_A)=\Gamma_A$ for every $A\subset I_n$ with $2\leq |A|\leq n-2$ (recall the definition of $\Gamma_A$ before Proposition \ref{prop:cup}). This implies that $\sigma(A)=A$ for every $A\subset I_n$ with $|A|\neq n/2$. In particular, since $n\geq 5$, this holds for every $A\subset I_n$ such that $|A|=2$. This clearly implies that $\sigma$ is the indentity.\par

Now we prove that every automorphism $f$ of $M_{0,n}^{\text{trop}}$ is induced by some permutation $\sigma\in S_n$. Let $\sigma$ be the permutation of $I_n$ as in the statement of Corollary \ref{cor:sigma3}. For $n\geq5$ all that is left to do is to prove that $\sigma$ satisfies $L(f(v))=\sigma(L(v))$ for all $\Gamma$ stable $n$-legged tree and $v\in V(\Gamma)$.\par
   We will prove the result by induction on the number of edges.  Let $\Gamma_1$ and $\Gamma_2$ be stable $n$-legged trees, such that $f(C(\Gamma_1))=C(\Gamma_2)$. Let $v$ be a leaf of $\Gamma_1$ and $e$ be the only edge attached to it. Contracting all edges of $\Gamma_1$ except $e$, and all edges of $\Gamma_2$ except $f(e)$, we get two stable $n$-legged trees $\Gamma_1'$ and $\Gamma_2'$  with exactly $2$ vertices. Let $v'$ be the vertex in $\Gamma_1'$ to which $v$ contracts. By  Remark \ref{rem:spec}, Corollary \ref{cor:sigma3} and the fact that $L(v)=L(v')$, we get that 
\begin{equation}
\label{eq:tau}
L(f(v)))=L(f(v'))=\sigma(L(v'))=\sigma(L(v)).
\end{equation}\par

	Contracting $e$ and using the induction hypothesis, we get that 
\begin{equation}
\label{eq:w}
L(f(w))=\sigma(L(w))\quad\text{for}\quad w\in V(\Gamma_1/\{e\}).
\end{equation}
 Let $v_1$ be the vertex in $\Gamma_1$ connected to $v$ and $w_1$ be the vertex in $\Gamma_1/\{e\}$ to which $v$ and $v_1$ contract. We have $L(w_1)=L(v)\cup L(v_1)$ and therefore, by Remark \ref{rem:spec}, $L(f(w_1))=L(f(v))\cup L(f(v_1))$. Combining Equations \eqref{eq:tau} and \eqref{eq:w}, we get $\sigma(L(w_1))=\sigma(L(v))\cup L(f(v_1))$, and since $\sigma(L(w_1))=\sigma(L(v))\cup\sigma(L(v_1))$, we have that $L(f(v_1))=\sigma(L(v_1))$. Finally, if $v_2\in V(\Gamma)\setminus\{v,v_1\}$, then there exists $w\in V(\Gamma/\{e\})$ such that $L(v_2)=L(w)$, and the result follows by Equation \eqref{eq:w}.\par
  For $n=4$, we have that $M_{0,4}^{trop}$ has just $3$ maximal cones which have dimension $1$ and they intersect along the $0$-dimensional cone. Therefore it is clear that its group of automorphisms is $S_3$. Note that the map $S_4\to S_3$, defined at the beginning of the proof, has kernel equal to the normal Klein subgroup $\{ (), (12)(34), (13)(24),(14)(23)\}$.
\end{proof}

\section{Final remarks and further questions}

	   We begin with an example in higher genus. The notion of automorphism of $M_{0,n}^{\text{trop}}$ naturally extends to $M_{g,n}^{\text{trop}}$ and $\overline{M}_{g,n}^{\textnormal{trop}}$ for $g>0$. Also Proposition \ref{prop:barra} should hold for higher genus. However,  a series of complications arise. For instance, Proposition \ref{prop:aut} does not hold for general graphs. Hence the typical cell of $M_{g,n}^{\text{trop}}$ is of type $C(\Gamma)=\mathbb{R}_{>0}^{|E(\Gamma)}/\Aut(\Gamma)$. Therefore, if $f$ is an automorphism of $M_{g,n}^{\text{trop}}$ such that $f(C(\Gamma))=C(\Gamma')$, then $f|_{C(\Gamma)}$ is induced by a bijection $f\col E(\Gamma)\to E(\Gamma')$ that is compatible with $\Aut(\Gamma)$ and $\Aut(\Gamma')$. Moreover, we do not see how to extend Proposition \ref{prop:count} for higher genus.\par
		
				\begin{figure}[h]
\[
\begin{xy} <30pt,0pt>:
(0,0)*{\scriptstyle\bullet}="a"; 
(1,0)*{\scriptstyle\bullet}="b";
(2.5,0)*{\scriptstyle\bullet}="c";
(3.5,0)*{\scriptstyle\bullet}="d";
(2.5,-2)*{\scriptstyle\bullet}="e";
(3.5,-2)*{\scriptstyle\bullet}="f";
(0.5,-2)*{\scriptstyle\bullet}="g";
%(-2,-2)*{\scriptstyle\bullet}="h"; 
%(-1,-2)*{\scriptstyle\bullet}="i";
%(-2,-4)*{\scriptstyle\bullet}="j"; 
%(-1,-4)*{\scriptstyle\bullet}="l";
(0.75,-4)*{\scriptstyle\bullet}="m";
(2.5,-4)*{\scriptstyle\bullet}="n";
(3.5,-4)*{\scriptstyle\bullet}="o";
(0.5,-6)*{\scriptstyle\bullet}="p";
"a"+0;"b"+0**\crv{"a"+(0.5,0.7)};
"a"+0;"b"+0**\crv{"a"+(0.5,-0.7)};
"a"+0;"b"+0**\crv{"a"+(0.5,0)};
"c"+0;"d"+0**\crv{"c"+(0.5,0)}; 
"c"+(-0.25,0)*\xycircle(0.25,0.25){-};
"d"+(0.25,0)*\xycircle(0.25,0.25){-};
"e"+0;"f"+0**\crv{"e"+(0.5,0)}; 
"e"+(-0.25,0)*\xycircle(0.25,0.25){-};
"g"+(-0.25,0)*\xycircle(0.25,0.25){-};
"g"+(0.25,0)*\xycircle(0.25,0.25){-};
%"h"+0;"i"+0**\crv{"h"+(0.5,0.7)};
%"h"+0;"i"+0**\crv{"h"+(0.5,-0.7)};
%"h"+0;"i"+0**\crv{"h"+(0.5,0)};
%"j"+0;"l"+0**\crv{"j"+(0.5,0.7)};
%"j"+0;"l"+0**\crv{"j"+(0.5,-0.7)};
%"j"+0;"l"+0**\crv{"j"+(0.5,0)};
"m"+(-0.25,0)*\xycircle(0.25,0.25){-};
"n"+0;"o"+0**\crv{"n"+(0.5,0)}; 
"a"+(0,0.44)*{{\Gamma_1}};
%"h"+(0.5,0.44)*{\scriptstyle{l_1}};
%"h"+(0.5,0.1)*{\scriptstyle{l_1+l_2}};
%"h"+(0.5,-0.44)*{\scriptstyle{l_1}};
%"j"+(0.5,0.44)*{\scriptstyle{l_1}};
%"j"+(0.5,0.1)*{\scriptstyle{l_1}};
%"j"+(0.5,-0.44)*{\scriptstyle{l_1}};
"c"+(0.5,0.44)*{{\Gamma_2}};
"e"+(1.4,0.2)*{{\Gamma_4}};
"g"+(-0.7,0.44)*{{\Gamma_3}};
"m"+(-0.75,0)*{{\Gamma_5}};
"n"+(1.4,0.2)*{{\Gamma_6}};
"a"+(-0.13,0)*{\scriptstyle{0}};
"b"+(0.13,0)*{\scriptstyle{0}};
"c"+(-0.13,0)*{\scriptstyle{0}};
"d"+(0.13,0)*{\scriptstyle{0}};
"e"+(-0.13,0)*{\scriptstyle{0}};
"f"+(0.1,0)*{\scriptstyle{1}};
"g"+(-0.13,0)*{\scriptstyle{0}};
%"h"+(-0.1,0)*{\scriptstyle{0}};
%"i"+(0.1,0)*{\scriptstyle{0}};
%"j"+(-0.1,0)*{\scriptstyle{0}};
%"l"+(0.1,0)*{\scriptstyle{0}};
"m"+(-0.1,0)*{\scriptstyle{1}};
"n"+(-0.1,0)*{\scriptstyle{1}};
"o"+(0.1,0)*{\scriptstyle{1}};
"p"+(-0.13,0)*{\scriptstyle{2}};
%{\ar^{\scriptstyle{l_3=0}}@{~}(0,-0.3)*{};(-1,-1.3)*{}};
%{\ar@{->}(-1,-1.3)*{};(-1.2,-1.5)*{}};
{\ar^{}@{~}(0.5,-0.7)*{};(0.5,-1.4)*{}};
{\ar@{->}(0.5,-1.4)*{};(0.5,-1.6)*{}};
{\ar^{}@{~}(2.2,-0.4)*{};(1.2,-1.4)*{}};
{\ar@{->}(1.2,-1.4)*{};(1,-1.6)*{}};
{\ar^{}@{~}(3,-0.3)*{};(3,-1.4)*{}};
{\ar@{->}(3,-1.4)*{};(3,-1.6)*{}};
{\ar^{}@{~}(0.5,-2.5)*{};(0.5,-3.4)*{}};
{\ar@{->}(0.5,-3.4)*{};(0.5,-3.6)*{}};
%{\ar_{\scriptstyle{l_2=0}}@{~}(-1.5,-2.6)*{};(-1.5,-3.2)*{}};
%{\ar@{->}(-1.5,-3.2)*{};(-1.5,-3.4)*{}};
{\ar^{}@{~}(3,-2.4)*{};(3,-3.4)*{}};
{\ar@{->}(3,-3.4)*{};(3,-3.6)*{}};
{\ar^{}@{~}(2.2,-2.4)*{};(1.2,-3.4)*{}};
{\ar@{->}(1.2,-3.4)*{};(1,-3.6)*{}};
%{\ar^{\scriptstyle{l_1=0}}@{~}(-1.2,-2.4)*{};(-0.2,-3.4)*{}};
%{\ar@{->}(-0.2,-3.4)*{};(0,-3.6)*{}};
{\ar^{}@{~}(0.5,-4.4)*{};(0.5,-5.4)*{}};
{\ar@{->}(0.5,-5.4)*{};(0.5,-5.6)*{}};
{\ar^{}@{~}(2.2,-4.4)*{};(1.2,-5.4)*{}};
{\ar@{->}(1.2,-5.4)*{};(1,-5.6)*{}};
%{\ar^{\scriptstyle{l_1=0}}@{~}(-1.2,-4.4)*{};(-0.2,-5.4)*{}};
%{\ar@{->}(-0.2,-5.4)*{};(0,-5.6)*{}};
\end{xy}
\]
\caption{Genus-$2$ tropical curves.}
\label{fig:g2}
\end{figure}
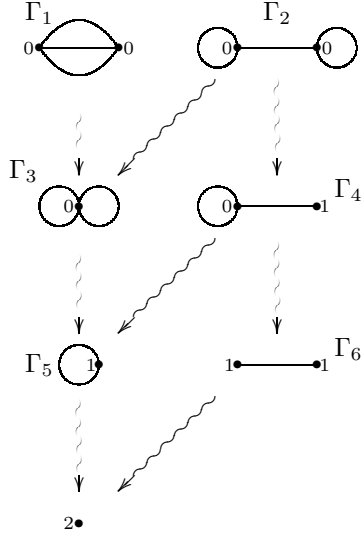

We now prove that the automorphism group of $M_2^{\text{trop}}$ is trivial.	Let $f$ be an automorphism of $M_2^{\text{trop}}$. Since $\Gamma_1$ only specializes to $\Gamma_3$, and $\Gamma_2$ specializes to $\Gamma_3$ and $\Gamma_4$, we see that $f(C(\Gamma_1))=C(\Gamma_1)$ and $f(C(\Gamma_2))=C(\Gamma_2)$ (the argument is essentially the one in Corollary \ref{cor:2f}). Analogously, it is easy to see that $f(C(\Gamma_i))=C(\Gamma_i)$ for $i=1,\ldots,6$. Since $\Aut(\Gamma_1)=S_3$, we see that $f|_{C(\Gamma_1)}$ is the identity, and hence $f|_{C(\Gamma_3)}$ and $f|_{C(\Gamma_5)}$ are the identity as well. If $f|_{C(\Gamma_2)}$ is the identity we are done. Otherwise, since there exists an automorphism of $\Gamma_2$ swapping the two loops, we have that $f$ must swap the unique edge of $\Gamma_2$ which is not a loop, with a loop of $\Gamma_2$. This would imply that $f(C(\Gamma_3))=C(\Gamma_4)$, a contradiction.\par
\medskip
   We conclude the paper by noting that, by \cite[Proposition 6.1.8]{ACP}, we have that a toroidal automorphism of a toroidal scheme $X$ induces an automorphism of its skeleton $\overline{\Sigma}(X)$, and that we get a group homomorphism
	\[
	\alpha_X\col\Aut^{\text{tor}}(X)\to \Aut(\overline{\Sigma}(X)),
	\]
	where $\Aut^{\text{tor}}(X)$ denotes the group of toroidal automorphisms of $X$. Hence, by \cite[Theorem 1.2.1]{ACP}, we have group homomorphisms
	\[
	\Aut(\overline{M}_{0,n})\hookleftarrow\Aut^{\text{tor}}(\overline{M}_{0,n})\to\Aut(\overline{M}_{0,n}^{\textnormal{trop}})
	\]
	which are, a posteriori, isomorphisms for $n\geq 5$ due to the result \cite{BM} of Bruno and Mella and Theorem \ref{thm:main}. Nevertheless it would be interesting to address the following questions
\begin{enumerate}
\item when the homomorphism	$\alpha_X$ is an isomorphism or at least injective?
\item	when the natural inclusion $\Aut^{\text{tor}}(X)<\Aut(X)$ is an isomorphism?
\end{enumerate}
We note that $\Aut^{\text{tor}}(\overline{M}_{0,4})=S_3$, while $\Aut(\overline{M}_{0,4})=\text{PGL}(2)$, hence Question (2) does not always have a positive answer.

\section*{Acknowledgements}
  We would like to thank Lucia Caporaso for precious suggestions on the preliminary versions of this paper.

\bigskip
\noindent{\smallsc Alex Abreu, Universidade Federal Fluminense,\\ Rua M. S. Braga, s/n, Valonguinho, 24020-005 Niter\'oi (RJ) Brazil.}\\
{\smallsl E-mail address: \small\verb?alexbra1@gmail.com?}

\bigskip
\bigskip

\noindent{\smallsc Marco Pacini, Universidade Federal Fluminense,\\ Rua M. S. Braga, s/n, Valonguinho, 24020-005 Niter\'oi (RJ) Brazil}\\
{\smallsl E-mail address: \small\verb?pacini@impa.br? and \small\verb?pacini.uff@gmail.com?}


\begin{thebibliography}{llll}

\bibitem{ACP}  D. Abramovich, L. Caporaso and S. Payne,
\emph{The tropicalization of the moduli space of curves},
Annales Scientifiques de L'\'Ecole Normale Sup\'erieure, series 4 {\bf 48 (4)} (2012), 765--809.

\bibitem{BMV} S. Brannetti, M. Melo and F. Viviani.
\emph{On the tropical Torelli Map},
Adv. Math. {\bf 226 (3)} (2011) 2546--2586.

\bibitem{BM} A. Bruno and M. Mella,
\emph{The automorphism group of $\overline{M}_{0,n}$},
Journal of the European Mathematical Society, {\bf 15 (3)} (2013), 949--968.

\bibitem{Caporaso} L. Caporaso, 
\emph{Algebraic and tropical curves: comparing their moduli spaces}.
In: Handbook of Moduli, Volume I. G. Farkas, I. Morrison (Eds.), Advanced Lectures in Mathematics, Volume XXIV (2012), 119--160.

\bibitem{Caporaso1} L. Caporaso,
\emph{Geometry of tropical moduli spaces and linkage of graphs}.
 Journal of Combinatorial Theory, Series A. {\bf 119} (2012) 579--598

\bibitem{CDPR} F. Cools, J. Draisma, S. Payne, E. Robeva,
\emph{A tropical proof of the Brill-Noehter Theorem}.
Adv. Math. {\bf 230} (2012) 759--776.

\bibitem{JP} D. Jensen and S. Payne,
\emph{Tropical independence I: Shapes of divisors and a proof of the Gieseker-Petri theorem.}
Algebra \& Number Theory, {\bf 8 (9)} (2014), 2043--2066.

\bibitem{JP1} D. Jensen and S. Payne,
\emph{Tropical indepence II: The maximal rank conjecture for quadrics}.
arxiv:1505.05460.

\bibitem{Ma} A. Massarenti,
\emph{The automorphism group of $\overline{M}_{g,n}$}.
Journal of the London Mathematical Society, {\bf 89} (2014), 131--150.

\bibitem{M}  G. Mikhalkin, 
\emph{Moduli spaces of rational tropical curves}.
Proceedings of G\"okova Geometry-Topology Conference 2006, 39--51, G\"okova Geometry/Topology Conference (GGT), G\"okova, 2007.

\end{thebibliography}
\end{document}